\documentclass[a4paper,10pt]{article}

\usepackage[utf8x]{inputenc}
\usepackage[english]{babel}

\usepackage{amsmath}
\usepackage{graphicx}
\usepackage{hyperref}

\usepackage{amsfonts}
\usepackage{amsthm}
\usepackage{amssymb}
\usepackage{cite} 

\newcommand{\R}{\mathbb{R}} 
\newcommand{\N}{\mathbb{N}}

\newtheorem{theorem}{Theorem}[section]
\newtheorem{lemma}[theorem]{Lemma}

\theoremstyle{definition}

\newtheorem{remark}[theorem]{Remark}

\numberwithin{equation}{section}

\title{Optimal shapes for positivity preserving\footnote{Research partially funded by the Deutsche Forschungsgemeinschaft (DFG, German Research
Foundation) - EI 1155/1}.}

 \author{Sascha Eichmann\thanks{The author thanks Prof. Reiner Sch\"atzle for fruitful discussions.}\\
 Mathematisch-Naturwissenschaftliche Fakultät,\\
 Eberhard Karls Universität Tübingen,\\
 Auf der Morgenstelle 10,\\
 D-72076 Tübingen, Germany\\
 E-Mail: \href{mailto:sascha.eichmann@math.uni-tuebingen.de}{sascha.eichmann@math.uni-tuebingen.de}}

\begin{document}
 \maketitle

\begin{abstract}
We are looking for an optimal convex domain on which the boundary value problem
$$\left\{\begin{array}{cc}(-\Delta)^2 u_\gamma-\gamma\Delta u_\gamma = f,& \mbox{ in }\Omega\\ u_\gamma=\partial_\nu u_\gamma=0,& \mbox{ on }\partial\Omega\end{array}\right.$$
admits a nonnegative solution for the most $\gamma$, if $f$ is a given nonnegative function.
\end{abstract}
\textbf{Keywords.} Positivity preserving, Boggio-Hadamard conjecture, shape optimisation\\ 
\textbf{MSC.} 35B09, 35B30, 35G15, 49Q10

\section{Introduction}
\label{sec:1}
The Boggio-Hadamard conjecture states, that for any convex domain $\Omega\subseteq\R^n$ any solution to
$$\left\{\begin{array}{cc}(-\Delta)^2 u = f,& \mbox{ in }\Omega\\ u=\partial_\nu u=0,& \mbox{ on }\partial\Omega\end{array}\right.$$
satisfies $u\geq 0$, if $f\geq 0$. Here $\nu$ is the outer unit normal of $\Omega$. 
This differential equation is also called the clamped plate equation (see e.g. \cite{PolyHarmBoundValue} for more  context). 
In it $\Omega$ is a thin plate clamped at the boundary and subject to an orthogonal force $f$. 
Then $u$ is the displacement of $\Omega$. Hence the conjecture can also be restated as:
$$\mbox{Does upward pushing imply upward bending?}$$
This conjecture was substantiated by Boggio by some explicit formulas for the Greens function in \cite{BoggioGreenFunction} in case $\Omega$ is a ball (see also the letters by Hadamard \cite{Hadamard1} and \cite{Hadamard2}). 
First counterexamples to this conjecture were given by Duffin in \cite{DuffinStrip} and a little later by Garabedian in \cite{GarabedianEllipse}. Today we even have counterexamples with constant $f$, see \cite{GrunauSweersSignChangeUniform1} and \cite{GrunauSweersSignChangeUniform2}.
On the other hand by e.g. conformal invariance one can use the ball to construct other domains with positivity preserving, see \cite{DallAcquaSweers2004}.
In general it is not known, what the driving force reagrding $\Omega$ behind positivity preserving or a lack thereof is, see e.g. the discussion in \cite[§ 5, § 6]{PolyHarmBoundValue} and the references therein.\\
For general elliptic differential operators of higher order positivity preserving cannot be expected, see  \cite{GrunauRomaniSweers2020}.

In \cite{CassaniTarsia} a related problem to \eqref{eq:1_0_1} has been discussed, i.e. whether for given $f\geq 0$ the boundary value problem
\begin{equation}
 \label{eq:1_0_1}
 \left\{\begin{array}{cc}(-\Delta)^2 u_\gamma-\gamma\Delta u_\gamma = f,& \mbox{ in }\Omega\\ u_\gamma=\partial_\nu u_\gamma=0,& \mbox{ on }\partial\Omega\end{array}\right.
\end{equation}
has a positive solution $u_\gamma$, if $\gamma$ is big enough. $\gamma$ is called tension and this kind of equation was first used by Bickley in \cite{Bickley} to model the membrane in microphones. 
Here the motivation differs: If $\gamma$ becomes bigger, the influence of the laplace term $-\Delta u_\gamma$ should increase. 
Since $-\Delta$ satisfies a maximum principle, i.e. also positivity preserving, the chances, that $u_\gamma$ is nonnegative increases if $\gamma$ increases.
This statement was proven rigorously in all dimensions in \cite{EichmannSchaetzlePosHighTens} and is cited in appendix \ref{sec:A} for the readers convenience (see also \cite{Grunau2002OneDimPosPres} for positivity preserving for all $\gamma\geq 0$ in dimension one).\\
Here we like to search for convex domains in which we have the 'most' positivity preserving, i.e. for the largest range of $\gamma\in\R$. 
Let us make the notion of our agenda more precise by letting
$R>0$, $f\in L^\infty(B_R(0))$, $c_1>0$, $f\geq0$, $n=2,3$. Now we consider the boundary value problem
\begin{equation}
 \label{eq:1_1}
 \left\{\begin{array}{cc}(-\Delta)^2 u_\gamma-\gamma\Delta u_\gamma = f,& \mbox{ in }\Omega\\ u_\gamma=\partial_\nu u_\gamma=0,& \mbox{ on }\partial\Omega\end{array}\right.
\end{equation}
for any domain $\Omega\subseteq B_R(0)\subseteq \R^n$.  
Since we like to solve this boundary value problem even for nonsmooth $\Omega$ we define $u_\gamma$ to be a (weak) solution of \eqref{eq:1_1}, if and only if $u_\gamma\in W^{2,2}_0(\Omega)$ and for all $\varphi\in C^\infty_0(\Omega)$ we have
\begin{equation}
 \label{eq:1_2}
 \int_\Omega \Delta u_\gamma\Delta\varphi + \gamma\nabla u_\gamma\cdot\nabla\varphi - f\varphi =0.
\end{equation}
By $\mu_1(\Omega)$ we will denote the buckling load, a kind of first eigenvalue, given by
\begin{equation}
 \label{eq:1_3}
 \mu_1(\Omega):=\inf_{u\in W^{2,2}_0(\Omega)-\{0\}}\frac{\int_\Omega|\Delta u|^2}{\int_\Omega|\nabla u|^2}.
\end{equation}
$\mu_1(\Omega)$ acts as a lower bound for $\gamma$, because upon reaching it, we would have nonuniqueness or nonexistence of the boundary value problem \eqref{eq:1_1}. 
See e.g. \cite[§ 3.2]{PolyHarmBoundValue} for more details regarding the buckling load.
Next we introduce the set in which we will minimise:
\begin{equation}
 \label{eq:1_4}
 M:=\{\Omega\subseteq B_R(0)|\ \Omega\mbox{ is a convex domain and }\mathcal{L}^n(\Omega)=c_1\}.
\end{equation}
We define an energy on $M$ by
\begin{equation}
 \label{eq:1_5}
 \gamma_f(\Omega):=\inf\{\gamma > -\mu_1(\Omega)|\ \forall \tilde{\gamma}\geq\gamma\mbox{ the weak solution to \eqref{eq:1_1} satisfies }u_{\tilde{\gamma}}\geq 0\mbox{ a.e.}\} 
\end{equation}
The result \cite[Thm. 1.1]{EichmannSchaetzlePosHighTens} (see also the appendix \ref{sec:A}) guarantees, that $\gamma_f$ is finite for bounded domains $\Omega$ with boundary in $C^4$.
This note is dedicated to proving the following theorem:
\begin{theorem}
 \label{1_1}
 There exists an $\Omega_f\in M$, such that for all $\Omega\in M$ we have
 $$\gamma_f(\Omega_f)\leq \gamma_f(\Omega).$$
\end{theorem}
This $\Omega_f$ is an optimal domain for positivity preserving, since the range for $\gamma\in\R$ is as large as possible. 
The strategy of the proof is as follows:
We use the direct variational method. Compactness for a sequence of domains will be achieved w.r.t. Hausdorff distance, see section \ref{sec:2}.
Lower-semicontinuity is at the heart of the proof, in which the convexity of the domains will play a crucial role, see Thm. \ref{2_2}.
The proof itself is given in section \ref{sec:3}. 
Finally we discuss some open questions in section \ref{sec:4}.

\section{Hausdorff distance and convexity}
\label{sec:2}
Here we collect some important results concerning the Hausdorff distance of sets in $\R^n$. Some of them are focused on convex sets. The Hausdorff distance is defined as
\begin{equation}
 \label{eq:2_1}
 d_H(A,B):=\inf\{\varepsilon>0|\ A\subseteq B_\varepsilon\mbox{ and }B\subseteq A_\varepsilon\}.
\end{equation}
Here
$$A_\varepsilon:=\bigcup_{x\in A} B_\varepsilon(x) = \{x\in\R^n|\ d(x,A) <\varepsilon\} .$$
For compact sets $d_H$ is a metric. Furthermore for all compact sets $K\subset \R^n$ the set
\begin{equation*}
 M_K:=\{\tilde{K}\subseteq K|\ \tilde{K}\mbox{ compact}\}
\end{equation*}
is compact with respect to $d_H$ (see e.g. \cite[§ 2.10.21]{Federer}).  Furthermore we have the following result for convex sets, which is also part of the famous Blaschke selection Theorem (see also \cite{BeerConvexMeasure74} for a different approach concerning the convergence of Lebesgue measure and again \cite[§ 2.10.21]{Federer} for the convexity of the limit).
We define convergence of convex domains, i.e. for open convex sets $\Omega_m,\Omega\subseteq\R^n$ as
\begin{equation}
\label{eq:2_1_2}
 d_H(\overline{\Omega}_m,\overline{\Omega})\rightarrow 0,\ m\rightarrow\infty.
\end{equation}

The result is then as follows:
\begin{lemma}[\cite{Blaschke1916} p.61 \& p.62]
 \label{2_1}
 Let $\Omega_m\subseteq\R^n$ be a sequence of convex domains converging w.r.t. to $d_H$ to $\Omega\subseteq\R^n$. Then
 $$\mathcal{L}^n(\Omega_m)\rightarrow \mathcal{L}^n(\Omega)$$
 and $\Omega$ is convex.
\end{lemma}

The next Theorem will greatly help in showing that the limit of some sequences considered in the proof of the main result will satisfy a differential equation weakly. The proof is inspired by \cite[Prop. 10]{Willis07}.
\begin{theorem}
 \label{2_2}
 Let $\Omega_m\subseteq \R^n$ be a sequence of convex open domains converging w.r.t. \eqref{eq:2_1_2} to an open convex $\Omega\subseteq \R^n$. Let $K\subset \Omega$ be compact in $\R^n$. Then there exists an $m_0\in\N$, such that for all $m\geq m_0$ we have
 $$K\subset \Omega_m.$$
\end{theorem}
\begin{proof} 
 We proceed by contradiction and assume after possibly choosing a subsequence, that
 $$\forall m\in\N\ \exists x_m\in K:\ x_m\notin \Omega_m.$$
 Since $K$ is compact, there exists an $r>0$ independent of $m$, such that
 $$B_r(x_m)\subseteq \Omega \mbox{ for all }m\in\N.$$
 Since $\Omega_m$ is convex, the Hahn-Banach seperation theorem, see e.g. \cite[Thm 3.4]{FARudin}, yields $L_m\in\partial B_1(0)$, $\alpha_m\in\R$, such that
 $$\Omega_m\subseteq \{y\in \R^n|\ \langle L_m,y\rangle \geq \alpha_m\}\mbox{ and } \langle L_m,x_m\rangle \leq\alpha_m.$$
 Here $\langle\cdot,\cdot,\rangle$ denotes the euclidean scalar product. By e.g. setting $y_m:=x_m -\frac{r}{2}L_m$ (see Figure \ref{2_1}), we obtain an $y_m\in B_r(x_m)\subseteq \Omega$ with
 $$d(y_m, \Omega_m)\geq\frac{r}{2}.$$
 \begin{figure}[h]
 \centering 
\includegraphics{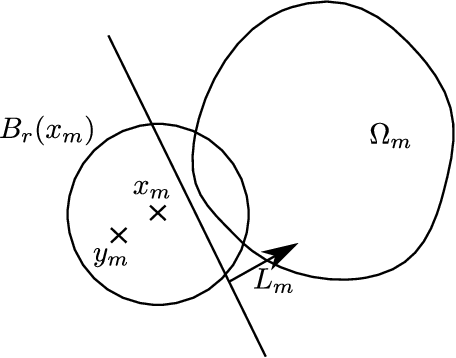}  
\caption{Seperating $x_m$ from $\Omega_m$ by Hahn-Banach.}
\label{fig:2_1}
\end{figure}
 
 Therefore
 $$y_m\notin (\Omega_m)_{\frac{r}{4}}.$$
 Since $y_m\in\Omega$, we have by the definition of Hausdorff distance
 $$d_H(\Omega,\Omega_m)\geq \frac{r}{4}.$$
This implies that $\Omega_m\nrightarrow \Omega$ w.r.t. Hausdorff distance, which is a contradiction.
 \end{proof}

\begin{remark}
 \label{2_3}
 Theorem \ref{2_2} fails without convexity. Let $\Omega$ be any open set and let $x\in \Omega$. Then we define the sequence
 $$\Omega_R:=\Omega-B_R(x).$$
 Then we have  (see also Figure \ref{fig:2_2})
 $$\Omega_R\rightarrow \Omega\mbox{ w.r.t. }d_H \mbox{ for }R\downarrow 0.$$

\begin{figure}[h]
 \centering 
\includegraphics[height=3cm]{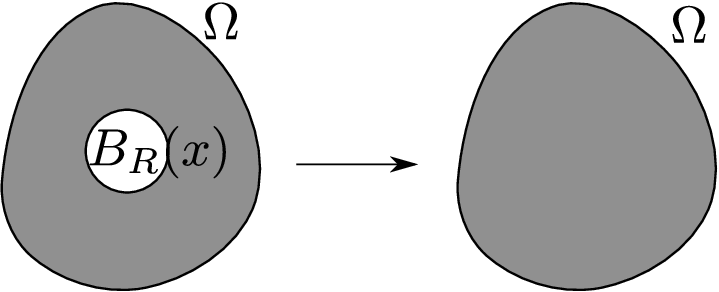}  
\caption{Counterexample for Thm. \ref{2_2} without convexity.}
\label{fig:2_2}
\end{figure} 
If we now choose any compact set $K\subset \Omega$ containing $x$, the result of Theorem \ref{2_2} cannot be expected.
\end{remark}

Let us also remark, that the boundary of a convex set is Lipschitz and therefore of Lebesgue measure zero:
\begin{lemma}
 \label{2_4}
 Let $C\subseteq\R^n$ be convex. Then $\partial C$ is a Lipschitz boundary and therefore
 $$\mathcal{L}^n(\partial C)=0.$$
\end{lemma}
\begin{proof}
 Since $C$ is convex, $\partial C$ is locally a graph of a convex function $f$. 
 $f$ is therefore also locally Lipschitz.
 By covering the boundary with at most countable many balls we then obtain 
 $\mathcal{L}^n(\partial C)=0$.
\end{proof}

\section{Proof of the main Theorem}
\label{sec:3}
\begin{proof}
 Let $\Omega_m\in M$ be a minimising sequence for $\gamma_f$, i.e.
 $$\gamma_f(\Omega_m)\rightarrow\inf_{\Omega\in M}\gamma_f(\Omega),\mbox{ for }m\rightarrow\infty.$$
 Then $\overline{\Omega}_m\subseteq \overline{B_R(0)}$ and after relabeling $\overline{\Omega}_m$ converges w.r.t. the Hausdorff distance $d_H$ to a compact set $K\subseteq \overline{B_R(0)}$. By Lemma \ref{2_1} $K$ is convex and $\mathcal{L}^n(K)=c_1>0$. Hence by Lemma \ref{2_4}
 $$\mathcal{L}^n(\partial K)=0$$
 and we set $\Omega_f:=\dot{K} = K-\partial K$, i.e. the topological inner part of $K$. Since the measure of it is positive, it is not empty. Furthermore it is still convex and satisfies
 $$\mathcal{L}^n(\Omega_f)=c_1,$$
 hence $\Omega_f\in M$. Now we need to show a lower semi-continuity estimate. For this let 
 $$\gamma > \inf_{\Omega\in M}\gamma_f(\Omega)$$
 be fixated. For $m$ big enough we have
 $$\gamma > \gamma_f(\Omega_m)\geq - \mu_1(\Omega_m),$$
 hence we have unique weak solutions $u_m\in W^{2,2}_0(\Omega_m)$ to \eqref{eq:1_1} on $\Omega_m$. We extend these by $0$ to the whole of $B_R(0)$. Therefore we have
 $$u_m\in W^{2,2}_0(B_R(0)).$$
 We find an $\alpha>0$ such that after choosing a subsequence and relabeling, we obtain
 $$\gamma-\alpha > \gamma_f(\Omega_m)\geq -\mu_1(\Omega_m).$$
 This yields
 $$\gamma - \alpha + \mu_1(\Omega_m)>0\mbox{ for all }m\in\N.$$
 If we now choose $\delta,\varepsilon'>0$ small enough, we also have
 $$(1-\delta)\mu_1(\Omega_m) + \gamma-\varepsilon'>0.$$
 Since $C^\infty_0(\Omega_m)$ is dense in $W^{2,2}_0(\Omega_m)$, the solution $u_m$ is an admissible test function. Hence Poincar\'{e}s and Youngs inequality yield
 \begin{align*}
  \int_{\Omega_m}(\Delta u_m)^2 + \gamma|\nabla u_m|^2 =& \int_{\Omega_m}f u_m\\
  \leq&C(\varepsilon,\operatorname{diam}(\Omega_m))\int_{\Omega_m}f^2 + \varepsilon\int_{\Omega}|\nabla u_m|^2
 \end{align*}
for every $\varepsilon>0$. Rearranging yields with the above choosen $\delta>0$
$$\delta \int_{\Omega_m}(\Delta u_m)^2 + (1-\delta)\int_{\Omega_m}(\Delta u_m)^2 + (\gamma - \varepsilon)\int_{\Omega_m}|\nabla u_m|^2 \leq C\int_{\Omega_m}f^2.$$
By the definition of $\mu_1(\Omega_m)$, see \eqref{eq:1_3}, we obtain
$$\delta \int_{\Omega_m}(\Delta u_m)^2 +((1-\delta)\mu_1(\Omega_m) + (\gamma-\varepsilon))\int_{\Omega_m}|\nabla u_m|^2\leq C(\varepsilon,\operatorname{diam}(\Omega_m))\int_{\Omega_m}f^2.$$
Choosing $ \varepsilon<\varepsilon'$ yields a constant $C=C(\varepsilon,R,f)>0$ such that
$$\|u_m\|_{W^{2,2}_0(B_R(0))}\leq C.$$
After choosing a subsequence and relabeling, we can assume there exists a $u_0\in W^{2,2}_0(B_R(0))$, such that
\begin{align*}
 u_m\rightarrow u_0&\mbox{ weakly in }W^{2,2}_0(B_R(0))\\
 u_m\rightarrow u_0&\mbox{ strongly in }W^{1,2}_0(B_R(0))\\
 u_m\rightarrow u_0&\mbox{ uniformely on }B_R(0).
\end{align*}
Here $n=2,3$ and the Sobolev embedding has been applied.
Since $u_m\geq 0$, the pointwise convergence yields $u_0\geq 0$ almost everywhere. Now we show, that $u_0$ satisfies \eqref{eq:1_2} on $\Omega_f$:\\
Let $\varphi\in C^\infty_0(\Omega_f)$ be arbitrary. Then
$$\operatorname{spt}(\varphi)\subset \Omega_f$$
is a compact set. By Theorem \ref{2_2} we find an $m_0\in\N$, such that for all $m\geq m_0$ we have
$$\operatorname{spt}(\varphi)\subset \Omega_m.$$
Since the $u_m$ are solutions to the equation, we therefore have for these $m$
$$0=\int_{B_R(0)}\Delta u_m\Delta\varphi + \gamma\nabla u_m\nabla \varphi - f\varphi.$$
Then the weak convergence in $W^{2,2}_0(B_R(0))$ yields
\begin{align*}
 0=&\int_{B_R(0)}\Delta u_0\Delta\varphi + \gamma\nabla u_0\nabla \varphi - f\varphi\\
 =&\int_{\Omega_f}\Delta u_0\Delta\varphi + \gamma\nabla u_0\nabla \varphi - f\varphi,
\end{align*}
hence $u_0$ satisfies \eqref{eq:1_2} on $\Omega_f$. Now we have to show that $u_0$ also satisfies the boundary values, i.e. $u_0\in W^{2,2}_0(\Omega_f)$. In a first step we show
$$\operatorname{spt}(u_0)\subset \overline{\Omega_f}.$$
Let $x\in B_R(0)$ such that $u_0(x)>0$. Then we find an $m_0(x)\in\N$ such that for all $m\geq m_0$ we have $u_m(x)>0$. In this argument the uniform convergence in the form of pointwise convergence everywhere is central. Hence $x\in \Omega_m$. The Hausdorff convergence yields for all $\varepsilon>0$ a $z_m\in\overline{\Omega_f}$ such that
$$x\in B_\varepsilon(z_m)$$
for $m$ big enough. Since $\overline{\Omega_f}$ is compact, we find a subsequence, such that $z_m\rightarrow z_\varepsilon\in\overline{\Omega_f}$. Furthermore
$$x\in \overline{B_\varepsilon(z_\varepsilon)}.$$
If we choose $\varepsilon:=\frac{1}{k}$ with $k\in\N$, we can again employ the compactness of $\overline{\Omega_f}$, which yields after choosing a subsequence $z_{\frac{1}{k}}\rightarrow z_0\in \overline{\Omega_f}$. By $x\in\overline{B_\frac{1}{k}(z_\frac{1}{k})}$ we have 
$$x=z_0\in \overline{\Omega_f}.$$
Since $x$ was choosen arbitrary, we have
$$\operatorname{spt}(u_0)\subseteq \overline{\Omega_f}.$$
Now we can show $u_0\in W^{2,2}_0(\Omega_f)$. Without loss of generality we assume $0\in \Omega_f$.  We define
$$v_\lambda(x):=u_0(\lambda\cdot x).$$
Since $\Omega_f$ is convex, we have for all $\lambda>1$
$$\operatorname{spt}(v_\lambda)\subsetneq \Omega_f.$$
This yields $v_\lambda\in W^{2,2}_0(\Omega_f)$. Furthermore
$$\|v_\lambda\|_{W^{2,2}_0(\Omega_f)}\leq \lambda^2\|u_0\|_{W^{2,2}_0(B_R(0))}.$$
Hence for $\lambda\downarrow 1$ we find a weakly convergent subsequence, such that
$$v_{\lambda_k}\rightarrow g\in W^{2,2}_0(\Omega_f).$$
Since the dimensions satisfy $n=2,3$, we can employ the Sobolev embedding theorem and by continuity of $u_0$ obtain
$$v_{\lambda_k}\rightarrow u_0\mbox{ pointwise everywhere.}$$
Without a restriction to the dimension, this step could also be proven by inner regularity theory, see e.g. \cite[§ 2.5]{PolyHarmBoundValue}.
Either way the limit $g$ is always $u_0$, irrespective of the choice of subsequence. This yields
$$v_\lambda\rightarrow u_0\mbox{ weakly in }W_0^{2,2}(\Omega_f)\mbox{ for }\lambda\downarrow 1\mbox{ and }u_0\in W^{2,2}_0(\Omega_f).$$
Hence $u_0\in W^{2,2}_0(\Omega_f)$ is a weak solution of \eqref{eq:1_1} w.r.t. our choosen $\gamma$ and $u_0\geq 0$. Therefore
$ \gamma_f(\Omega_f) \leq \gamma. $
Since $\gamma>\inf_{\Omega\in M}\gamma_f(\Omega)$ arbitrary, we have
$$\gamma_f(\Omega_f)\leq \inf_{\Omega\in M}\gamma_f(\Omega) = \lim_{m\rightarrow\infty}\gamma_f(\Omega_m).$$
Hence $\Omega_f$ is a minimiser in $M$.

As an addendum we can also show, that $\inf_{\Omega\in M}\gamma_f(\Omega)\geq-\mu_1(\Omega_f)$. Let $v_1\in W^{2,2}_0(\Omega_f)$ be such that
$$\int_{\Omega_f} |\nabla v_1|\,  =1\mbox{ and }\mu_1(\Omega_f)=\int_{\Omega_f}|\Delta v_1|^2\, .$$
Then there exists a sequence $\varphi_k\in C^\infty_0(\Omega_f)$, such that
$$\varphi_k\rightarrow v_1\mbox{ in }W^{2,2}_0(\Omega_f).$$
Since the support of $\varphi_k$ is compact and in $\Omega_f$ we can apply Theorem \ref{2_2} and find for every $k\in\N$ an $m_0(k)\in\N$, such that for all $m\geq m_0$ 
$$\varphi_k\in C^\infty_0(\Omega_m).$$
Hence
$$\mu_1(\Omega_f)=\lim_{k\rightarrow\infty} \frac{\int_{B_R(0)}|\Delta\varphi_k|^2}{\int_{B_R(0)}|\nabla\varphi_k|^2}\geq \limsup_{m\rightarrow\infty}\mu_1(\Omega_m).$$
Therefore
$$-\mu_1(\Omega_f)\leq-\limsup_{m\rightarrow\infty}\mu_1(\Omega_m)\leq \limsup_{m\rightarrow\infty}\gamma_f(\Omega_m)=\inf_{\Omega\in M}\gamma_f(\Omega)$$
and the result follows.
\end{proof}

\section{Open problems}
\label{sec:4}
In this section we discuss some open problems and possible further lines of inquiry:
\begin{itemize}
 \item Numerics: The proof of \cite[Thm. 1.1]{EichmannSchaetzlePosHighTens} is indirect, i.e. it yields no direct control on $\gamma_f$. Hence providing a sound numerical ansatz for calculating $\gamma_f$ seems out of reach at the moment. Finding such an ansatz would be highly beneficial for further research though.
 \item Interpolation: One possible way of attacking the numerical problem, would be to show the following interpolation property: Given $\Omega\subseteq\R^n$, $f>0$ and $\gamma_1 < \gamma_2$ such that the solution to \eqref{eq:1_1} $u_{\gamma_i}$ is nonnegative ($i=1,2$). 
 Then prove or disprove that for all $\gamma\in]\gamma_1,\gamma_2[$ the solution $u_\gamma$ to \eqref{eq:1_1} is also nonnegative.
 \item Regularity: Completely open is the question of higher regularity of the minimiser $\Omega_f$ of Thm. \ref{1_1}. 
 Since $\Omega_f$ is convex, the boundary is rectifiable but it is unclear if it is smooth. The answer will probably depend on properties of $f$. 
 The hard part is actually finding a good sense of Euler-Lagrange equation for $\Omega_f$, if such a thing even exists. 
 One could possibly try to perform a shape derivative, but since it is unclear how $\Omega\rightarrow\gamma_f(\Omega)$ behaves under small perturbations of $\Omega$, this is not trivial. 
 As a first pointer one may want to look at small smooth perturbations of the ball, see e.g. \cite[§ 6.1.1]{PolyHarmBoundValue} and references therein.
 \item Examples: It should be quite useful to be able to actually calculate some explicit values for $\gamma_f$ for example for an ellipsoid and constant $f$. In dimension one this has been done in \cite{Grunau2002OneDimPosPres}. 
 In higher dimension there is a useful result for the ball in \cite[Thm. 1.4]{LaurencotWalker2015}, although it is not readily applicable, because the assumptions $c_i\leq 0$ in \cite[Eq. 1.12]{LaurencotWalker2015} may be a problem. 
 Similarly \cite[Prop. 2.1]{OmraneKhenissy2014} is not directly applicable.
 \item Symmetry breaking: The Boggio-Hadamards conjecture was substantiated by the ball. If $f$ is radially symmetric, or even constant, one may wonder, if this symmetry transfers over to a minimiser $\Omega_f$. 
 If one chooses $f=0$, then always $\gamma_f(\Omega)=-\mu_1(\Omega)$. One would therefore look for a maximiser of the buckling load, but it is conjectured, that the minimiser is the ball, see e.g. \cite[§3.2]{PolyHarmBoundValue}. Therefore the present author expects a symmetry breaking phenomenon.
 \item Convexity: The proof of Theorem \ref{1_1} heavily relies on the domains $\Omega\in M$ being convex. If we drop this assumption, it is unclear whether such a minimiser still exists. 
 On the other hand by \cite[Thm 1.1]{EichmannSchaetzlePosHighTens} $\Omega\rightarrow\gamma_f(\Omega)$ is still well defined, if $\partial \Omega\in C^4$, even without $\Omega$ being convex. 
 Here the hard part seems to be to actually obtain a suitable control in some suitable topology for a minimising sequence, which allows for a similar result like Theorem \ref{2_2}.
\end{itemize}

\appendix
\section{Positivity for large tension}
\label{sec:A}
By the following theorem $\gamma_f$ is well defined under some regularity conditions.
\begin{theorem}[see Thm. 1.1 in \cite{EichmannSchaetzlePosHighTens}]
 \label{A_1} Let $\Omega\subset\subset \R^n$ be connected and open with $\partial\Omega\in C^4$. Furthermore let $\tau>0$. 
 Then there exists a $\gamma_0=\gamma_0(\Omega,\tau)>0$ such that for all $f\in L^\infty(\Omega)$ with $f\geq 0$ and
 $$\int f\, d\mathcal{L}^n \geq \tau \|f\|_{L^\infty(\Omega)} > 0,$$
 we have that the solution $u_\gamma$ of
 $$\left\{\begin{array}{cc} (-\Delta)^2u_\gamma - \gamma\Delta u_\gamma = f,& \mbox{ in }\Omega\\u=\partial_\nu u=0,&\mbox{ on }\partial\Omega\end{array}\right.$$
 satisfies 
 $$u_\gamma>0\mbox{ in }\Omega\mbox{ for all }\gamma_0\leq \gamma < \infty.$$
\end{theorem}

\phantomsection
\addcontentsline{toc}{section}{Literatur}
\bibliography{bibliography.bib}

\begin{thebibliography}{10}

\bibitem{BeerConvexMeasure74}
G.~A. {Beer}.
\newblock {Hausdorff metric and convergence in measure}.
\newblock {\em Michigan Math. J.}, 21(1):63--64, 1974.
\newblock \url{https://doi.org/10.1307/mmj/1029001209}.

\bibitem{Bickley}
W.~G. {Bickley}.
\newblock Deflexions and vibrations of a circular elastic plate under tension.
\newblock {\em Phil. Mag.}, (7) 15:776--797, 1933.
\newblock \url{https://doi.org/10.1080/14786443309462222}.

\bibitem{Blaschke1916}
W.~{Blaschke}.
\newblock {\em {Kreis und Kugel}}.
\newblock Leipzig:: Verlag von Veit \& Comp., 1916.
\newblock \url{https://doi.org/10.1515/9783111506937}.

\bibitem{BoggioGreenFunction}
T.~Boggio.
\newblock Sulle funzioni di green d'ordine.
\newblock {\em m. Rend. Circ. Mat. Palermo}, 20:97--135, 1905.
\newblock \url{https://doi.org/10.1007/BF03014033}.

\bibitem{CassaniTarsia}
D.~{Cassani} and A.~{Tarsia}.
\newblock Maximum principle for higher order operators in general domains.
\newblock {\em Advances in Nonlinear Analysis}, 11(1):655--671, 2021.
\newblock \url{https://doi.org/10.1515/anona-2021-0210}.

\bibitem{DallAcquaSweers2004}
A.~{Dall'Acqua} and G.~{Sweers}.
\newblock On domains for which the clamped plate system is positivity
  preserving.
\newblock {\em Differential Equations and Inverse Problems, ed. by Carlos
  Conca, Raul Manasevich, Gunter Uhlmann and Michael Vogelius, AMS, 2004},
  2004.
\newblock \url{https://doi.org/10.1090/conm/362/06609}.

\bibitem{DuffinStrip}
R.~J. {Duffin}.
\newblock {On a question of Hadamard concerning super-biharmonic functions}.
\newblock {\em Journal of Mathematics and Physics}, 27:253--258, 1948.
\newblock \url{https://doi.org/10.1002/sapm1948271253}.

\bibitem{EichmannSchaetzlePosHighTens}
S.~{Eichmann} and R.~{Schätzle}.
\newblock {Positivity for the clamped plate equation under high tension}.
\newblock {\em Annali di Matematica}, 201:2001--2020, 2022.
\newblock \url{https://doi.org/10.1007/s10231-022-01188-9}.

\bibitem{Federer}
Herbert {Federer}.
\newblock {\em {Geometric Measure Theory}}.
\newblock Springer-Verlag Berlin, 1969.
\newblock \url{https://doi.org/10.1007/978-3-642-62010-2}.

\bibitem{GarabedianEllipse}
P.~R. {Garabedian}.
\newblock A partial differential equation arising in conformal mapping.
\newblock {\em Pacific J. Math}, 1:485--524, 1951.

\bibitem{PolyHarmBoundValue}
F.~{Gazzola}, H.-Chr. {Grunau}, and G.~{Sweers}.
\newblock {\em {Polyharmonic Boundary Value Problems}}.
\newblock Springer, 2010.
\newblock {Lecture Notes in Mathematics 1991}
  \url{https://doi.org/10.1007/978-3-642-12245-3}.

\bibitem{Grunau2002OneDimPosPres}
H.-Chr. {Grunau}.
\newblock {Positivity, change of sign and buckling eigenvalues in a
  one-dimensional fourth order model problem}.
\newblock {\em Adv. Differential Equations}, 7(2):177--196, 2002.
\newblock \url{https://doi.org/10.57262/ade/1356651850}.

\bibitem{GrunauRomaniSweers2020}
H.-Chr. {Grunau}, G.~{Romani}, and G.~{Sweers}.
\newblock Differences between fundamental solutions of general higher order
  elliptic operators and of products of second order operators.
\newblock {\em Math. Ann.}, 381:1031--1084, 2020.
\newblock \url{https://doi.org/10.1007/s00208-020-02015-3}.

\bibitem{GrunauSweersSignChangeUniform1}
H.-Chr. {Grunau} and G.~{Sweers}.
\newblock A clamped plate with a uniform weight may change sign.
\newblock {\em Discrete \& Continuous Dynamical Systems - S}, 7(4):761--166,
  2014.
\newblock \url{https://doi.org/10.3934/dcdss.2014.7.761}.

\bibitem{GrunauSweersSignChangeUniform2}
H.-Chr. {Grunau} and G.~{Sweers}.
\newblock In any dimension a "clamped plate" with a uniform weight may change
  sign.
\newblock {\em Nonlinear Analysis A: T.M.A}, 97:119--124, 2014.
\newblock \url{https://doi.org/10.1016/j.na.2013.11.017}.

\bibitem{Hadamard1}
J.~{Hadamard}.
\newblock M\'{e}moire sur le probl\`{e}me d'analyse relatif \`{a}
  l'\'{e}quibilibre des plaques \'{e}lastiques encastr\'{e}es.
\newblock {\em {\OE}uvres de Jaques Hadamard, Tome II, CNRS Paris}, pages
  515--641, 1968.
\newblock Reprint of: M\'{e}moire pr\'{e}sent\'{e}s par divers savant a
  l'Acad\'{e}mie des Sciences (2), 33:1-128, 1908.

\bibitem{Hadamard2}
J.~Hadamard.
\newblock Sur certains cas int\'{e}ressants du probl\`{e}me biharmonique.
\newblock {\em {\OE}uvres de Jaques Hadamard, Tome III, CNRS Paris}, pages
  1297--1299, 1968.
\newblock Reprint of: Atti IV Congr. Intern. Mat. Rome 12-14, 1908.

\bibitem{OmraneKhenissy2014}
{Hanen Ben Omrane} and {Sa\"{i}ma Khenissy}.
\newblock {Positivity preserving results for a biharmonic equation under
  Dirichlet boundary conditions}.
\newblock {\em Opuscula Math.}, 34(3):601--608, 2014.
\newblock \url{http://dx.doi.org/10.7494/OpMath.2014.34.3.601}.

\bibitem{LaurencotWalker2015}
P.~{Lauren\c{c}ot} and C.~{Walker}.
\newblock {Sign-preserving property for some fourth-order elliptic operators in
  one dimension or in radial symmetry}.
\newblock {\em JAMA}, 127:69--89, 2015.
\newblock \url{https://doi.org/10.1007/s11854-015-0024-2}.

\bibitem{FARudin}
W.~{Rudin}.
\newblock {\em {Functional Analysis}}.
\newblock Mc Graw Hill, 1973.

\bibitem{Willis07}
M.D. {Willis}.
\newblock {Hausdorff Distance and Convex Sets}.
\newblock {\em Journal of Convex Analysis}, 14(1):109--117, 2007.

\end{thebibliography}
\bibliographystyle{plain}

% %\newpage
% \noindent Sascha Eichmann\\
% Mathematisch-Naturwissenschaftliche Fakultät,\\
% Eberhard Karls Universität Tübingen,\\
% Auf der Morgenstelle 10, D-72076 Tübingen, Germany\\
% E-mail: \href{mailto:sascha.eichmann@math.uni-tuebingen.de}{sascha.eichmann@math.uni-tuebingen.de}

\end{document}